\documentclass[12pt]{amsart}
\usepackage{a4wide,graphicx}
\usepackage{color}
\usepackage{caption}
\usepackage{comment}
\usepackage{enumerate}
\usepackage{upgreek}
\usepackage[left=1in, right=1in]{geometry}
\usepackage{bbm,yfonts}
\usepackage{amsmath,hyperref}  
\usepackage[utf8]{inputenc}

\let\pa\partial

\newcommand{\R}{{\mathbb R}}

\newcommand{\K}{{\mathcal K}}

\newcommand{\renyi}{{\mathcal{R}}}
\newcommand{\dd}{{\mathrm{d}}}

\newcommand{\E}{\mathcal{E}}
\newcommand{\M}{\Theta}
\newcommand{\J}{\mathcal{J}}
\newcommand{\I}{\mathcal{I}}
\newcommand{\entropy}{\mathcal{H}}

\newcommand{\renep}{\mathcal{N}}

\theoremstyle{plain}
\newtheorem{theorem}{Theorem}[section]   
\newtheorem{lemma}[theorem]{Lemma}   
\newtheorem{proposition}[theorem]{Proposition}

\theoremstyle{definition}
\newtheorem{definition}{Definition}[section]

\theoremstyle{remark}
\newtheorem{remark}{Remark}[section]

\date{version of \today}

\begin{document}

\title[Improved equilibration rates to self-similarity]{Improved equilibration rates to self-similarity\\ for strong solutions of\\ a thin-film and related evolution equations}
 
\author{Mario Bukal$^1$}
\address[1]{University of Zagreb,
Faculty of Electrical Engineering and Computing\newline
Unska 3, 10000 Zagreb, Croatia}
\email{mario.bukal@fer.hr}

\thanks{This work has been supported by the Croatian Science
Foundation under project IP-2022-10-2962, and by the University of Zagreb Faculty of Electrical Engineering and Computing under project DEEPWAVE}



\begin{abstract}
This paper investigates the asymptotic behavior of strong solutions to a family of nonlinear fourth-order evolution equations on the real line, with particular focus on the thin-film equation $\partial_tu = -(uu_{xxx})_x$. The method builds on the framework introduced by Carrillo and Toscani ({\em Nonlinearity} 27 (2014), 3159) for second-order nonlinear diffusion equations --- by introducing a time-dependent rescaling that preserves the second moment, we establish sharp convergence rates toward the steady state in terms of the relative R\'enyi entropy. Compared to rates derived from the dissipation of the classical relative entropy, this approach yields improved estimates at early and intermediate times, and consequently a sharper convergence in the $L^1$-norm.
The method is developed at a formal level for the family of fourth-order equations, including the well-known Derrida-Lebowitz-Speer-Spohn (DLSS) equation, but can be rigorously justified for strong solutions of the thin-film equation.
\end{abstract} 

\date{\today}

\maketitle

\section{Introduction}
In this paper we consider the asymptotic behavior of solutions to the Cauchy problem for a nonlinear fourth-order evolution equation
\begin{equation}\label{1.eq:tfe}
    \pa_t u = -\left(uu_{xxx}\right)_{x}\,, \quad x\in\R\,,\ t>0\,,
\end{equation}
given a nonnegative initial data 
\begin{equation}\label{1.eq:id}
    u(x,0) = u_0(x)\geq 0\,,\quad x\in\R\,
\end{equation}
of unit mass and of finite second moment. This is an old and important question that has been studied repeatedly in the literature. Briefly, Smyth and Hill were the first to explicitly calculate self-similar source-type solutions to \eqref{1.eq:tfe} \cite{SmHi88}. After the development of the fundamental theory \cite{BBD95, Ber96, BeFr90, BePu96} for equation \eqref{1.eq:tfe}, Carrillo and Toscani proved the convergence of strong solutions towards the unique source-type solution with an explicit and universal algebraic rate \cite{CaTo02}. Later, this result has been extended to the notion of weak solutions in the framework of the $L^2$-Wasserstein gradient flows \cite{MMS09}. 
In this paper, we improve the result obtained by Carrillo and Toscani in \cite{CaTo02}. 

Equation \eqref{1.eq:tfe} appeared in \cite{CDGKSZ93} as a simplified model describing a two-dimensional droplet breakup in a Hele-Shaw cell and has since received significant attention in the mathematical community. The graph of the unknown function $u(x,t)$ describes the dynamics of the interface between a thin fluid neck of thickness $2u(x,t)$ and the surrounding air. The equation is naturally studied within a wider class of thin-film equations 
\begin{equation}\label{1.eq:tfen}
    \pa_t u = -\left(u^nu_{xxx}\right)_{x}\,, 
\end{equation}
parametrized by a positive parameter $n$, formally describing the degree of slip of fluid films. The case of $n=3$ corresponds to the no-slip condition on the fluid velocity at the substrate, $n=2$ to the Navier slip condition, while $n=1$ (our case) corresponds to the full slip or the Darcy flow \cite{Ber98, Mye98, ODB97}.
The first rigorous existence analysis of equations \eqref{1.eq:tfen} was performed by Bernis and Friedman \cite{BeFr90}. Further development of the theory discussing the regularity of weak solutions, the long-term behavior, the finite speed of propagation, and other fine properties under different boundary conditions emerged in \cite{BeGr05, BBD95, Ber96, BePu96, GiOt01}. 

Besides \eqref{1.eq:tfen}, equation \eqref{1.eq:tfe} belongs to another parametrized family of fourth-order equations, 
\begin{align}\label{1.eq:WGF}
    \pa_t u = -
    \left(u\left(u^{p-3/2}(u^{p-1/2})_{xx}\right)_x\right)_x\,
\end{align}
for $p=3/2$. This family first appeared in \cite{DeMc08} in connection with the second-order porous medium equations  
\begin{equation*}
    \pa_tu = \left(u^p\right)_{xx}\,.
\end{equation*}
Some explicit solutions have been constructed and their short- and long-term be\-ha\-vi\-or was discussed. Afterwards, this connection has been deepen in \cite{MMS09} by justifying the $L^2$-Wasserstein gradient flow structure of \eqref{1.eq:WGF} and obtaining optimal equilibration rates of constructed weak solutions to the self-similar solution.
Let us point out that the family \eqref{1.eq:WGF} has another prominent and well-studied member, namely the Derrida-Lebowitz-Speer-Spohn (DLSS) equation \cite{DLSS91}, obtained for $p=1$ 
\begin{equation*}
    \pa_t u = -\left(u\left(\frac{(\sqrt{u})_{xx}}{\sqrt{u}}\right)_x\right)_x\,.
\end{equation*}
In fact, this form of the DLSS equations appears in the context of local approximations of quantum drift-diffusion models \cite{DMR05}. Its gradient flow structure has been justified in \cite{GST09} and the basic existence theory on bounded domains accompanied by qualitative properties developed in \cite{BLS94, CCT05, Fisch13, Fisch14, JuMa08, JuPi00}.

The family of equations \eqref{1.eq:WGF} can be formally (for smooth and positive solutions) written in a symmetric form \cite{MMS09}
\begin{equation}\label{1.eq:lnp_sym}
    \pa_t u = -\frac{1}{p-1}\left(u^p\left(u^{p-1}\right)_{xx}\right)_{xx}\,,\quad x\in\R\,,\ t>0\,,\quad\text{for } p\neq 1\,,
\end{equation}
which will be the focus of the present paper. We will formally consider the asymptotic behavior of solutions to \eqref{1.eq:lnp_sym} with initial data $u_0$ of unit mass and of finite second moment. For $p=1$, \eqref{1.eq:lnp_sym} turns into the original logarithmic form of the DLSS equation \cite{DLSS91}
\begin{equation}\label{1.eq:dlss_log}
    \pa_t u = -\left(u\left(\ln u\right)_{xx}\right)_{xx}\,, 
\end{equation}
which has been recently formally interpreted as another gradient flow structure related to the diffusive transport \cite{MRSS25, MRS25}. Eventually, equations \eqref{1.eq:lnp_sym} for $1\leq p <2$ were analyzed in \cite{Buk16}, and the existence of global nonnegative weak solutions under periodic boundary conditions was proved together with their exponential equilibration to the constant steady state.

Now we turn back to the question of the long-term behavior of solutions to equation \eqref{1.eq:tfe} and \eqref{1.eq:lnp_sym} in general, and provide a brief literature review. Smyth and Hill \cite{SmHi88} formally calculated a self-similar attractor given by
\begin{equation}\label{1.def:sssts}
    U(x,t) = \frac{1}{24\, t^{1/5}}\left(C - \frac{x^2}{t^{2/5}}\right)_+^2\,,
\end{equation}
where $C$ is a positive constant determined by the unit mass constraint and $(\cdot)_+$ denotes the positive part of a quantity. Analogously, self-similar solutions of \eqref{1.eq:lnp_sym} (in fact \eqref{1.eq:WGF}) were calculated in \cite{DeMc08},
\begin{equation}\label{1.def:self-sim-sol}
    U(x,t) = \frac{a_p}{t^{1/(2p+2)}}\left(C - (p-1)\frac{x^2}{t^{1/(p+1)}}\right)_+^{\frac{1}{p-1}}\,,\quad a_p = \left(2p(2p-1)\right)^{-1/(2p-2)}\,,\quad p\neq 1\,.
\end{equation}
For $p=1$, the latter turns into the corresponding Gaussian density. 
Note that self-similar solutions \eqref{1.def:sssts} and \eqref{1.def:self-sim-sol} are properly rescaled Barenblatt profiles describing the long-term behavior of solutions to the porous medium equation \cite{Vaz07}.


The thin-film equation \eqref{1.eq:tfe} in symmetric form reads 
\begin{equation*}
    \pa_t u = -2\left(u^{3/2}\left(u^{1/2}\right)_{xx}\right)_{xx}\,.
\end{equation*} 
A distinguishing feature, not shared by other members of the family \eqref{1.eq:lnp_sym}, is that this reformulation can be made rigorous for strong solutions of the thin-film equation \cite{BBD95, BePu96}.
Precisely that was the key point in \cite{CaTo02} to prove the convergence of strong solutions to self-similar solutions \eqref{1.def:sssts} in the $L^1$-norm at an algebraic decay $O(t^{-1/5})$. 
We briefly outline the main steps of the proof. Employing a self-similar rescaling of equation \eqref{1.eq:tfe} in terms of 
\begin{equation} \label{1.def:rescaling1}
    w(x,t) = e^t u(e^tx, (e^{5t}-1)/5)\,,\qquad x\in\R\,,\ t>0\,.
\end{equation}
leads to a Fokker-Planck-type equation
\begin{equation}\label{1.eq:tf_resc}
\pa_t w = -2\left(w^{3/2}\left(w^{1/2}\right)_{xx}\right)_{xx} + (xw)_x\,.    
\end{equation}
The evolution of the relative entropy along strong solutions to \eqref{1.eq:tf_resc} reads as
\begin{align}\label{1.eq:Hev}
    \frac{\dd}{\dd t} \entropy(w(t)|B) = - D(w)\,,
\end{align}
where $\entropy(w|B) = \entropy(w) - \entropy(B)$ with 
\begin{equation*}
    \entropy(w) = \int_{\R}\left(\frac{x^2}{2}w + \sqrt{\frac{8}{3}}w^{3/2}\right)\dd x\,,\qquad B(y) = \frac{1}{24}\left(C - y^2\right)_+^2\,,
\end{equation*}
and the entropy dissipation is given by
\begin{equation*}
    D(w) = \int_{\R\cap \{w>0\}} w\left(\frac{x^2}{2} + \sqrt{6}w^{1/2}\right)_x^2\dd x + \frac{2}{\sqrt{6}}\int_{\R\cap \{w>0\}} w^{3/2}\left(\frac{x^2}{2} + \sqrt{6}w^{1/2}\right)_{xx}^2\dd x\,.
\end{equation*}
The identity \eqref{1.eq:Hev} is rigorously justified for strong solutions of \eqref{1.eq:tf_resc}. Relating the relative entropy to the entropy dissipation via a generalized logarithmic-Sobolev inequality,
\begin{equation*}
    \entropy(w|B) \leq \frac{1}{2}D(w)\,,
\end{equation*}
leads to the exponential decay of the relative entropy. Finally, utilizing the Csisz\'ar-Kullback inequality \cite{CJMTU01} and returning to the original variables finishes the proof. 

Analogous result was obtained in \cite{MMS09} for weak solutions of \eqref{1.eq:WGF}, but with methods from the gradient flow theory. Namely, the weak solutions constructed in \cite{MMS09} do not possess sufficient regularity to perform explicit calculations using the symmetric form \eqref{1.eq:lnp_sym}.
In \cite{McSe15}, a rescaled version of the equation \eqref{1.eq:WGF} was formally linearized around its steady state and using the known results for the porous medium equations a complete asymptotic expansion of solutions to \eqref{1.eq:WGF} has been conjectured for long times. 
Exponential convergence to the steady state of weak solutions to the rescaled thin-film equation in a weighted higher-order (energy) norm measuring smoothness and localization was obtained in \cite{CaUl14}. Finally, we quote \cite{Gna15}, where the global existence, uniqueness, and asymptotic stability for the corresponding free boundary problem of the thin-film equation was established and proved that small perturbations of self-similar solution \eqref{1.def:sssts} converge to \eqref{1.def:sssts} with improved decay rates.

In the present paper we explore yet another idea by Carrillo and Toscani \cite{CaTo14}. Instead of rescaling \eqref{1.def:rescaling1}, which preserves only mass and the first moment of solutions, in \cite{CaTo14} they introduced a novel rescaling of porous medium equations that also preserves the second moment in time. Such a rescaling then leads to a nonlocal Fokker-Planck-type equations. Furthermore, instead of the usual Newman-Ralston relative entropies \cite{New84}, they consider the evolution of relative R\'enyi entropies and invoke the concavity property of the R\'enyi entropy power \cite{SaTo14} to obtain the evolution inequality for the relative R\'enyi entropy. The obtained inequality leads to a super-exponential convergence to the corresponding steady state. For early and intermediate times, this result improves the rate of convergence given in terms of the usual relative entropy and thus, improves the rate of convergence in the $L^1$-norm for porous medium equations. 
Here we formally follow this method for the family of equations \eqref{1.eq:lnp_sym} for $1\leq p \leq 3/2$ and obtain a formal improvement of the result from \cite{MMS09}. Formally, since the regularity of weak solutions from \cite{MMS09} is insufficient to justify all the steps of the method. However, as in \cite{CaTo02}, the method can be made rigorous for strong solutions of \eqref{1.eq:tfe}, and thus, we improve the result of \cite{CaTo02}.

The paper is organized as follows. In Section \ref{sec:2} we perform basic calculations: derive nonlocal rescaled equations, calculate steady states, introduce R\'enyi entropies and collect key functional inequalities needed for the method. Section \ref{sec:3} is devoted to the proof of our main result formulated in Theorem \ref{3.tm}. The justification of formal calculations for the thin-film equation is carried out in Section \ref{sec:thin-film}, and Section \ref{sec:dlss} underlines the result for the DLSS equation.

\section{Nonlocal equations, steady states and R\'enyi entropies}\label{sec:2}

Let us write the family of equations \eqref{1.eq:lnp_sym} (including \eqref{1.eq:dlss_log}) in a more compact form
\begin{equation}\label{2.eq:lnp}
    \pa_t u = -\left(u^p\left(\ln_p u\right)_{xx}\right)_{xx}\,, \quad x\in\R\,,\ t>0\,,
\end{equation}
where $\ln_p$ denotes a $p$-logarithm given by
\begin{equation*}
    \ln_p u = \frac{u^{p-1} - 1}{p-1}\,,\quad \text{for }p\neq 1\,.
\end{equation*}
For $p=1$, $\ln_1 u = \ln u$ denotes the natural logarithm.
\subsection{Change of variables}
Consider the second moment
\begin{equation*}
    \M(u) = \int_\R x^2u\, \dd x
\end{equation*}
and its evolution along solutions to \eqref{2.eq:lnp}
\begin{align*}
    \frac{\dd }{\dd t}\M(u(t)) &= -\int_\R x^2\left(u^p\left(\ln_{p}u\right)_{xx}\right)_{xx} \dd x =
    -2\int_\R u^p \left(\ln_{p}u\right)_{xx} \dd x\\ \nonumber
    &= 2p \int_\R u\left(\ln_{p}u\right)_{x}^2\dd x =: 4p\, \I_p(u)\,.
\end{align*}
This defines a generalization of the Fisher information $\I_p(u)$.
Above, we formally used the identity (for $p\neq 1$)
\begin{equation*}
    (u^p)_x =  p\, u\left(\frac{u^{p-1} - 1}{p-1}\right)_{x} = p\;  u\left(\ln_{p}u\right)_{x}\,.
\end{equation*}

Define $v(y,\tau)$ by the following self-similar rescaling of $u(x,t)$ \cite{CaTo14}:
\begin{equation}\label{1.def:v}
    u(x,t) = \left(\frac{\M(u(t))}{\M_0}\right)^{-1/2} v \left(x\left(\frac{\M(u(t))}{\M_0}\right)^{-1/2}, \tau(t)\right),
\end{equation}
where $\M_0 = \int_\R x^2 u_0(x)\dd x$ is the second moment of the initial data and the rescaled time is given by
\begin{equation}\label{1.def:tau}
    \tau(t) = \frac{\M_0}{4p}\ln\left(\frac{\M(u(t))}{\M_0}\right).
\end{equation}
Then the change of variables reveals that
\begin{equation*}
    \M(u(t)) = \int_\R x^2u(x,t)\, \dd x = \frac{\M(u(t))}{\M_0} \int_\R y^2 v(y,\tau)\dd y\,.
\end{equation*}
Hence, for every $\tau > 0$
\begin{equation*}
   \M(v(\tau)) = \int_\R y^2 v(y,\tau)\dd y = \M_0\,.
\end{equation*}
\begin{proposition}
    Let $u(x,t)$ be a solution to equation \eqref{2.eq:lnp}, then function $v(y,\tau)$ defined by \eqref{1.def:v} satisfies a nonlocal fourth-order equation
    \begin{equation}\label{2.eq:v_p}
        \pa_\tau v = -\frac{1}{\I_p(v)}\left(v^p\left(\ln_{p}v\right)_{yy}\right)_{yy} + \frac{2p}{\M_0}(y v)_y\,.
    \end{equation}
\end{proposition}
\begin{proof}
Let us briefly write $\M_t = \M(u(t))$ and $y = x (\M_t/\M_0)^{-1/2}$. Then we calculate
    \begin{align*}
        \pa_t u(x,t) &= -\frac12\left(\frac{\M_t}{\M_0}\right)^{-3/2}\frac{\dd \M_t}{\dd t}\frac{1}{\M_0}v(y,\tau) + \left(\frac{\M_t}{\M_0}\right)^{-1/2}\left(\pa_\tau v(y,\tau) \frac{\dd \tau}{\dd t} + \pa_yv\,\frac{\dd y}{\dd t}\right)\\
        &= -\frac{1}{2\M_0}\left(\frac{\M_t}{\M_0}\right)^{-3/2}\frac{\dd \M_t}{\dd t}v + \left(\frac{\M_t}{\M_0}\right)^{-1/2}\left(\pa_\tau v\, \frac{\dd \tau}{\dd t} - \frac{x}{2\M_0} \left(\frac{\M_t}{\M_0}\right)^{-3/2}\frac{\dd \M_t}{\dd t}\pa_y v\right)\\
        &= -\frac{1}{2\M_0} \left(\frac{\M_t}{\M_0}\right)^{-3/2}\frac{\dd \M_t}{\dd t}\pa_y(y v) + \left(\frac{\M_t}{\M_0}\right)^{-1/2}\pa_\tau v\, \frac{\dd \tau}{\dd t}\,.
    \end{align*}
    Next,
    \begin{align*}
        \I_p(u) = \left(\frac{\M_t}{\M_0}\right)^{-p}\I_p(v)\,,
    \end{align*}
    which implies
    \begin{equation*}
        \frac{\dd \M_t}{\dd t} = 4p \left(\frac{\M_t}{\M_0}\right)^{-p}\I_p(v)\,.
    \end{equation*}
    Finally,
    \begin{align*}
        \left(u^p\left(\ln_{p}u\right)_{xx}\right)_{xx} = \left(\frac{\M_t}{\M_0}\right)^{-3/2 -p}\left(v^p\left(\ln_{p}v\right)_{yy}\right)_{yy}\,.
    \end{align*}
    Putting all together
    \begin{align*}
        \left(\frac{\M_t}{\M_0}\right)^{p + 1}\tau'(t)\pa_\tau v = - \left(v^p\left(\ln_{p}v\right)_{yy}\right)_{yy} + \frac{2p}{\M_0}\I_p(v)(y v)_y\,.
    \end{align*}
    Now choosing the time scale satisfying
    \begin{align*}
        \tau'(t) &= \left(\frac{\M_t}{\M_0}\right)^{-p - 1}\I_p(v)\,,\\
        \tau(0) &= 0\,,
    \end{align*}
    leads to \eqref{1.def:tau} and the fourth-order equation \eqref{2.eq:v_p}.
\end{proof}
\begin{remark}
 Scaling \eqref{1.def:v} preserves the second moment of rescaled solutions, but it can also be directly verified that equation \eqref{2.eq:v_p} also preserves the mass and the first moment.
\end{remark}

\subsection{Steady state}
Here we recall a connection of equations \eqref{2.eq:lnp} with the porous medium equation 
\begin{equation}\label{2.eq:PM}
    \pa_t u = (u^p)_{xx}\,.
\end{equation}
We follow the strategy of Carrillo and Toscani in \cite{CaTo02}. 
Let us add and subtract a second-order term 
\begin{equation*}
    \frac{2p}{\M_0}\frac{c_p}{\I_p(v)}\left(v^p\right)_{yy}
\end{equation*}
to equation \eqref{2.eq:v_p}, where $c_p>0$ will be determined later. Since 
\begin{equation*}
    \frac{2p}{\M_0}\frac{c_p}{\I_p(v)}\left(v^p\right)_{yy} = \frac{c_p}{\I_p(v)}\left(v^p\left(\frac{p}{\M_0}y^2\right)_{yy}\right)_{yy} = \frac{2p^2}{\M_0}\frac{c_p}{\I_p(v)}\left(v\left(\ln_pv\right)_{y}\right)_{y}\,,
\end{equation*} 
equation \eqref{2.eq:v_p} can be rewritten as
\begin{align*}
\pa_\tau v = -\frac{c_p}{\I_p(v)}\left(v^p\left(\frac{1}{c_p}\ln_pv + \frac{p}{\M_0}y^2\right)_{yy}\right)_{yy} + \left(v\left(\frac{2p^2}{\M_0}\frac{c_p}{\I_p(v)}\ln_pv + \frac{p}{\M_0}y^2\right)_y\right)_y\,.    
\end{align*}
Now choosing $c_p>0$ such that
\begin{equation}\label{2:def_cp}
    \frac{1}{c_p} = \frac{2p^2}{\M_0}\frac{c_p}{\I_p(v)}\,,
\end{equation}
asserts that equation \eqref{2.eq:v_p} has the same steady state as the rescaled porous medium equation
\begin{equation}\label{2:eq_PM2}
    \pa_\tau v = \left(v\left(\frac{2p^2}{\M_0}\frac{c_p}{\I_p(v)}\ln_pv + \frac{p}{\M_0}y^2\right)_y\right)_y = D(v^p)_{yy} + \frac{2p}{\M_0}(yv)_y\,.
\end{equation}
Note that \eqref{2:def_cp} has a unique positive solution 
\begin{equation*}
    c_p = \frac{\sqrt{\I_p(v)\M_0}}{\sqrt{2}p},
\end{equation*}
hence, the above diffusion coefficient $D$ amounts to
\begin{equation*}
    D = \frac{\sqrt{2}}{\sqrt{\M_0\I_p(v)}}\,.
\end{equation*}

Next, we take a look for steady states $V(y)$ of equation \eqref{2:eq_PM2}. A sufficient condition for a steady state is that $V(y)$ satisfies
\begin{equation}\label{2:eq_ss1}
    V\left(\frac{\sqrt{2}p}{\sqrt{\M_0\I_p(V)}}\ln_pV + \frac{p}{\M_0}y^2\right)_y = 0\,.
\end{equation}
This equation can be solved implicitly and one obtains
\begin{equation*}
    V(y) = \left(C_p - \frac{\sqrt{\I_p(V)}}{\sqrt{2\M_0}}(p-1)y^2\right)_+^{1/(p-1)},\qquad\text{for }p\ne1\,,
\end{equation*}
where $(\cdot)_+ = \max\{\cdot,0\}$ and $C_p$ is a positive constant which fixes the mass of $V$ to unity. Clearly,
\begin{equation*}
    \left(V^{p-1}\right)_y^2 = \frac{2\I_p(V)}{\M_0}(p-1)^2y^2\,,
\end{equation*}
then multiplying this identity by $V$ and integrating with respect to $y$ yields
\begin{equation*}
    \frac{1}{2}\int_{\R}V\left(\frac{V^{p-1}}{p-1}\right)_y^2\dd y = \frac{\I_p(V)}{\M_0}\int_{\R}y^2 V\,\dd y\,.
\end{equation*}
Since the expression on the left-hand side equals $\I_p(V)$, then
\begin{equation*}
    \int_{\R}y^2 V\,\dd y = \M_0\,,
\end{equation*}
i.e.~the steady state $V(y)$ has the same second moment as $v(y,\tau)$.
However, this steady state is not unique. Namely, given any positive number $\sigma>0$, it is straightforward to check that a function defined by 
\begin{equation*}
    V_\sigma(y) = \frac{1}{\sigma^{3/2}}V\left(\frac{y}{\sqrt{\sigma}}\right)
\end{equation*}
is another steady state, i.e.~solves equation \eqref{2:eq_ss1} and has the same second moment $\M_0$. We now choose $\sigma$ satisfying
\begin{equation}
    \label{2.eq:sigma}
    \sigma = \frac{\sqrt{\M_0}}{p\sqrt{2\I_p(V_\sigma)}}\,,
\end{equation}
which uniquely determines $\sigma$.
Observe that $\I_p(V_\sigma) = \sigma^{-3p+1}\I_p(V)$ and a function $\phi_p(\sigma)$ defined by 
\begin{equation*}
    \phi_p(\sigma) = p\sigma\sqrt{2\I_p(V_\sigma)} = p\sqrt{2\I_p(V)} \sigma^{-3(p-1)/2}
\end{equation*}
is for $p>1$ strictly decreasing, and satisfies  $\lim_{\sigma\downarrow0}\phi_p(\sigma) = +\infty$ and $\lim_{\sigma\to+\infty}\phi_p(\sigma) = 0$. Thus, \eqref{2.eq:sigma} has the unique solution. For this $\sigma$, let $B_p(y)$ be the unique solution of the equation
\begin{equation}\label{2.eq:ss2}
    p\left(\ln_p B_p\right)_y + \frac{y}{\sigma} = 0\,.
\end{equation}
By the choice of $\sigma$, $B_p(y)$ also solves the steady state equation \eqref{2:eq_ss1}, hence, its second moment equals $\M_0$. On the other hand, equation \eqref{2.eq:ss2} has an explicit solution, namely the Barenblatt profile, which is for $p>1$ given by
\begin{equation}\label{2.eq:Bar_p}
    B_p(y) = \left(C_p - \frac{p-1}{p}\frac{y^2}{2\sigma}\right)_+^{1/(p-1)}\,.
\end{equation}
To conclude, with $\sigma$ defined by \eqref{2.eq:sigma} and constant $C_p>0$ controlling the mass, the Barenblatt profile \eqref{2.eq:Bar_p} is the unique steady state of equation \eqref{2.eq:v_p}.
\begin{remark}
    For $p=1$ an implicit solution of equation \eqref{2:eq_ss1} satisfies 
    \begin{equation*}
        V(y) = C\exp\left({-\sqrt{\frac{\I(V)}{2\M_0}}y^2}\right)
    \end{equation*}
    and analogous conclusions hold. The choice of $\sigma$ reduces to $\sigma = \M_0$ and the Barenblatt profile becomes the Gaussian density
    \begin{equation*}
        B_1(y) = C_1\exp(-y^2/(2\sigma))\,.
    \end{equation*}
\end{remark}

\subsection{R\'enyi entropies}
R\'enyi entropy of order $p$ ($p\neq 1$) of a probability density $f$ is defined by \cite{CoTh91,Ren60}
\begin{equation}\label{2.def:ren_p}
    \renyi_p(f) = \frac{1}{1-p}\ln\left(\int_\R f^p\, \dd x\right)\,.
\end{equation}
It is a generalization of the Boltzmann-Shannon entropy \cite{CoTh91, Sha48}
\begin{equation*}
    \entropy(f) = -\int_\R f\ln f\, \dd x\,,
\end{equation*}
which can be formally recovered on the limit as $p\to 1$ in \eqref{2.def:ren_p}.

In the following we collect some results from the literature which will be used later in the analysis. 

\begin{lemma}[Lemma 4.1 in \cite{CaTo14}]
    Let $p>1/3$, $f$ be an arbitrary probability density of a bounded second moment, and let $B_{p,f}$ be the Barenblatt profile of the same second moment as $f$. Then
    \begin{equation}\label{2.ren_ineq}
        \renyi_p(B_{p,f}) - \renyi_p(f) \geq 0\,.
    \end{equation}
\end{lemma}
\noindent Inequality \eqref{2.ren_ineq} allows one to define a nonnegative quantity 
\begin{equation*}
    \renyi_p(f|B_{p,f}) = \renyi_p(B_{p,f}) - \renyi_p(f)\,,
\end{equation*}
called {\em the relative R\'enyi entropy} \cite{CaTo14}. 

Another important notion in information theory is the entropy power \cite{Sha48}. For R\'enyi entropies, the corresponding entropy power has been defined by \cite{SaTo14}
\begin{equation*}
    \renep_p(f) = \exp\left({(p+1)\renyi_p(f)}\right)\,.
\end{equation*} 
In connection with the porous medium equation \eqref{2.eq:PM}, the entropy power has a remarkable concavity property \cite{SaTo14}, i.e.
\begin{equation*}
    \frac{\dd^2}{\dd t^2}\renep_p(u(t)) \leq 0\,
\end{equation*}
along solutions to \eqref{2.eq:PM} for $p>-1$. The concavity property is essentially a consequence of the following functional inequality.
\begin{lemma}\label{2.lem:villp}Let $p>1/3$. 
For sufficiently smooth, positive and rapidly decaying probability densities $f$, it holds that
    \begin{equation}\label{2.ineq:KIp}
        \int_{\R}f^p\left(\ln_p f\right)_{xx}^2 \dd x\geq 4p^2\frac{\I_p^2(f)}{\int_\R f^p\, \dd x}\,.
    \end{equation}
\end{lemma}
\begin{proof} The proof of a more general result can be found in the proof of \cite[Theorem 1]{SaTo14}. Here we demonstrate \eqref{2.ineq:KIp} by using the Villani's trick from \cite{Vil00}. For every $\lambda\in\R$ it holds
    \begin{equation*}
        \int_\R f^p\left(\left(\ln_p f\right)_{xx} + \lambda\right)^2\dd x \geq0\,.
    \end{equation*}
    Expanding the left hand side and integrating by parts it follows
    \begin{align*}
         \int_{\R}f^p\left(\ln_p f\right)_{xx}^2 - 2\lambda p\I_p(f) + \lambda^2 \int_\R f^p\, \dd x \geq 0\,.
    \end{align*}
    Choosing 
    \begin{equation*}
        \lambda = \frac{2p}{\int_\R f^p\, \dd x}\I_p(f)
    \end{equation*}
    leads to \eqref{2.ineq:KIp}.
\end{proof}
Let us introduce the modified Fisher information as defined in \cite{CaTo14}
\begin{equation}\label{2.def:modF}
    \J_p(f) :=\frac{1}{\int_{\R}f^p\,\dd x}\int_{\{f>0\}}\frac{(f^p)_x^2}{f}\dd x = \frac{2p^2}{\int_\R f^p\, \dd x}\I_p(f)\,.
\end{equation}
The concavity of the R\'enyi entropy power implies the following inequality \cite{CaTo14,SaTo14}.
\begin{lemma}\label{2.lem:epip}
    Let $p>1/3$. Then for any probability density $f$ it holds
    \begin{equation}\label{2.ineq:epi_p}
        \exp\left((p+1)\renyi_{p}(f|B_{\sigma,f})\right) \leq \frac{\J_p(f)}{\J_p(B_{\sigma,f})}\,.
    \end{equation}
\end{lemma}
\begin{remark}
    As pointed out in \cite{CaTo14, SaTo14} inequality \eqref{2.ineq:epi_p} is a sub-family of the Gagliardo-Nirenberg-Sobolev inequalities.
    Specifically, the relative R\'enyi entropy satisfies
    \begin{equation}\label{2.eq:expreny}
        \exp\left(\renyi_p(f|B_{\sigma,f})\right) = \frac{\left(\int_{\R}B_{\sigma,f}^p\,\dd x\right)^{1/(1-p)}}{\left(\int_{\R}f^p\,\dd x\right)^{1/(1-p)}}\,.
    \end{equation}
    Thus, \eqref{2.ineq:epi_p} can be transformed for $p>1$ into
    \begin{equation*}
        \left(\int_{\R}f^p\,\dd x\right)^{2p/(p-1)}\leq \frac{\left(\int_{\R}B_{\sigma,f}^p\,\dd x\right)^{2p/(p-1)}}{\I_p(B_{\sigma,f}^p)}\I_p(f)\,.
    \end{equation*}
    Taking $g = f^{p-1/2}$ and using that $f=g^{2/(2p-1)}$ is a probability density, a straightforward calculation yields
    \begin{equation*}
        \|g\|_{L^{2p/(2p-1)}(\R)}\leq C_p\|g_x\|_{L^2(\R)}^{(2p^2-3p+1)/2p^2}\|g\|_{L^{2/(2p-1)}(\R)}^{(3p-1)/2p^2}\,,
    \end{equation*}
    which is a Gagliardo-Nirenberg inequality with explicitly given constant depending on the corresponding Barenblatt profile
    \begin{equation*}
        C_p = \left(\frac{2\left(\int_{\R}B_{\sigma,f}^p\,\dd x\right)^{2p/(p-1)}}{(2p-1)^2\I_p(B_{\sigma,f}^p)}\right)^{(2p^2-3p+1)/4p^2}.
    \end{equation*}
\end{remark}
Finally, we recall a connection of the relative R\'enyi entropy with a standard relative entropy functional, also called Newman-Ralston relative entropy \cite{New84}, which is defined by
\begin{equation*}
    \entropy_p(f|B_{\sigma,f}) = \frac{1}{p-1}\int_{\R}\left(f^p - B_{\sigma,f}^p - pB_{\sigma,f}^{p-1}(f - B_{\sigma,f})\right)\dd x\,.
\end{equation*}
\begin{lemma}\label{lem:NR-ren}
    Let $p>1$, then the relative Newman-Ralston entropy is bounded from above by the relative R\'enyi entropy in the following way
    \begin{equation*}
        \entropy_p(f|B_{\sigma,f}) \leq \left(\int_{\R}f^p\,\dd x\right)\renyi_p(f|B_{\sigma,f})\,.
    \end{equation*}
\end{lemma}
\noindent This result and its proof can be found in \cite[Lemma 4.2]{CaTo14}.

\begin{remark}
    Following the previous result, all decay estimates obtained for the relative R\'enyi entropy will accordingly apply to the the relative Newman-Ralston entropy and subsequently to the $L^1$-distance by means of the Csiszar-Kullback inequality \cite{CJMTU01}.
\end{remark}

\section{Improved equilibration rates to self-similarity} \label{sec:3}

\subsection{The family of fourth-order equations for $1< p \leq 3/2$}\label{sec:3.p}

In this section we formally calculate the evolution of the relative R\'enyi entropy along smooth non-negative spatially fast decaying solutions to equation \eqref{2.eq:v_p} and bound it by the relative R\'enyi entropy itself. Here we assume $1 < p \leq 3/2$. The case $p=1$ (the DLSS equation) will be treated separately in Section \ref{sec:dlss}, and formal calculations for the thin-film equation ($p=3/2$) will be justified for strong solutions in Section \ref{sec:thin-film}.

The evolution of the R\'enyi entropy along solutions to \eqref{2.eq:lnp} is calculated as follows:
\begin{align*}
    \frac{\dd}{\dd t}\renyi_p(u(t)) &= \frac{p}{1-p}\frac{1}{\int_\R u^p\, \dd x}\int_\R u^{p-1}\pa_t u\, \dd x
    = -\frac{p}{1-p}\frac{1}{\int_\R u^p\, \dd x}\int_\R u^{p-1}\left(u^p\left(\ln_{p}u\right)_{xx}\right)_{xx}\, \dd x\\
    &= \frac{p}{\int_\R u^p\, \dd x}\int_\R u^p\left(\frac{u^{p-1}-1}{p-1}\right)_{xx}^2\dd x =: \frac{p}{\int_\R u^p\, \dd x}\K_p(u)\,.
\end{align*}
This defines the second-order functional $\K_p(u)$.
Let $B_p$ denotes the Barenblatt profile \eqref{2.eq:Bar_p} with $\sigma$ satisfying \eqref{2.eq:sigma}. Then the evolution of the relative R\'enyi entropy along solutions to \eqref{2.eq:v_p} is given by
\begin{align*}
    \frac{\dd}{\dd \tau}\renyi_p(v(\tau)|B_p) &= - \frac{\dd}{\dd \tau}\renyi_p(v(\tau)) = -\frac{p}{1-p}\frac{1}{\int_\R v^p \dd y} \int_\R v^{p-1}\pa_\tau v\, \dd y\\
    &=  \frac{p}{1-p}\frac{1}{\int_\R v^p \dd y} \frac{1}{\I_p(v)} \int_\R v^{p-1}\left(v^p\left(\ln_{p}v\right)_{yy}\right)_{yy}\, \dd y \\
    &\qquad - \frac{p}{1-p}\frac{1}{\int_\R v^p \dd y} \frac{2p}{\M_0} \int_\R v^{p-1}(y v)_y\dd y\\
    &= -\frac{1}{\int_\R v^p \dd y} \frac{p}{\I_p(v)} \int_\R v^p\left(\ln_{p}v\right)_{yy}^2\, \dd y + \frac{2p}{\M_0}\,.
\end{align*}
In calculating the integral of the convective term we used integration by parts and obtained
\begin{align*}
    - \frac{p}{1-p}\frac{1}{\int_\R v^p \dd y} \frac{2p}{\M_0} \int_\R v^{p-1}(y v)_y\dd y &=
    \frac{p}{1-p}\frac{1}{\int_\R v^p \dd y} \frac{2p}{\M_0}(p-1) \int_\R y v^{p-1}v_y \,\dd y\\
    &= - \frac{1}{\int_\R v^p \dd y} \frac{2p}{\M_0} \int_\R y (v^{p})_y \,\dd y = \frac{2p}{\M_0}\,.
\end{align*}
Therefore, we have
\begin{align}\label{3.eq:rendiss2}
    \frac{\dd}{\dd \tau}\renyi_p(v(\tau)|B_p) &=  -\frac{1}{\int_\R v^p\, \dd y} \frac{p}{\I_p(v)} \K_p(v) + \frac{2p}{\M_0}\,.
\end{align}
Multiplying equation \eqref{2.eq:ss2} by $yB_p$, integrating over space and integrating by parts gives
\begin{equation*}
    \int_{\R}B_p^p\,\dd y = \frac{1}{\sigma}\int_{\R}y^2B_p\,\dd y = \frac{\M_0}{\sigma}\,.
\end{equation*}
Then combining the definition of $\sigma$ from \eqref{2.eq:sigma} with the latter identity yields
\begin{equation}\label{3.eq:2p}
    \frac{2p}{\M_0} = \frac{4p^3 \I_p(B_p)}{\left(\int_{\R}B_p^p\, \dd y\right)^2}\,.
\end{equation}
Using the inequality \eqref{2.ineq:KIp} from Lemma \ref{2.lem:villp} and the previous \eqref{3.eq:2p} in \eqref{3.eq:rendiss2} leads to
\begin{align}\nonumber
    \frac{\dd}{\dd \tau}\renyi_p(v(\tau)|B_p) &\leq -\frac{4p^3\I_p(v)}{\left(\int_{\R}v^p\, \dd y\right)^2} + \frac{4p^3\I_p(B_p)}{\left(\int_{\R}B_p^p\, \dd y\right)^2} = -2p\left(\frac{\J_p(v)}{\int_{\R}v^p\, \dd y} - \frac{\J_p(B_p)}{\int_{\R}B_p^p\, \dd y}\right)\\
    &= -2p\frac{\J_p(B_p)}{\int_{\R}B_p^p\, \dd y}\left(\frac{\J_p(v)}{\J_p(B_p)}\frac{\int_{\R}B_p^p\, \dd y}{\int_{\R}v^p\, \dd y} - 1\right) = -\frac{2p}{\M_0}\left(\frac{\J_p(v)}{\J_p(B_p)}\frac{\int_{\R}B_p^p\, \dd y}{\int_{\R}v^p\, \dd y} - 1\right).\label{3.ineq:rendiss3}
\end{align}
In the above, we used the modified Fisher information defined in \eqref{2.def:modF}. Using the entropy power inequality \eqref{2.ineq:epi_p} (Lemma \ref{2.lem:epip}) and identity \eqref{2.eq:expreny} it follows 
\begin{equation*}
    \frac{\J_p(v)}{\J_p(B_p)}\frac{\int_{\R}B_p^p\, \dd y}{\int_{\R}v^p\, \dd y}\geq \exp\left((p+1)\renyi_p(v|B_p) - (p-1)\renyi_p(v|B_p)\right) = \exp(2\renyi_p(v|B_p))\,.
\end{equation*} 
Therefore, going back to \eqref{3.ineq:rendiss3}, the dissipation of the relative R\'enyi entropy can be bounded with the relative R\'enyi entropy itself according to
\begin{align}\label{3.ineq:rendiss4}
    \frac{\dd}{\dd \tau}\renyi_{p}(v(\tau)|B_p) &\leq -\frac{2p}{\M_0}\left(\exp(2\renyi_p(v|B_p)) - 1\right)\,.
\end{align}
This differential inequality can be solved explicitly by separation of variables (cf.~\cite[Section 3]{CaTo14}) and one obtains the following decay estimate
\begin{equation*}
    \renyi_{p}(v(\tau)|B_p) \leq -\frac{1}{2}\ln\left(1 - \left(1 - \exp({-2\renyi_p(v_0|B_p)})\right)\exp\left(-\frac{4p}{\M_0}\tau\right)\right)\,,\quad \tau > 0\,.
\end{equation*}
\begin{figure}
    \includegraphics[width=6.5cm]{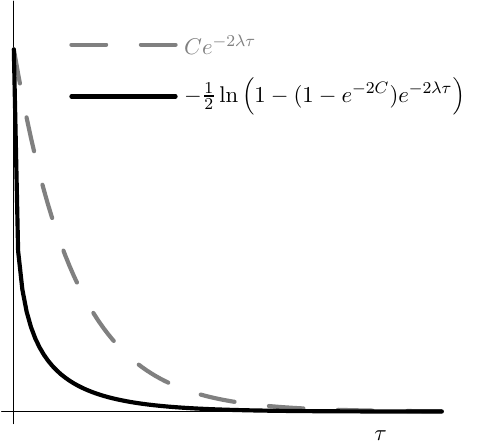}
    \caption{Comparison of exponential and super-exponential decay of re\-la\-tive entropies.}
    \label{fig:decay}
\end{figure}
Employing Lemma \ref{lem:NR-ren}, the relative Newmann-Ralston entropy has the same decay estimate 
\begin{align}\label{3.ineq:NR_ren}
    \entropy_p(v(\tau)|B_p) \leq C\left(\int_{\R}v^p\,\dd y\right)\renyi_p(v|B_p)\leq C\left(\int_{\R}v_0^p\,\dd y\right)\renyi_p(v|B_p)\,.
\end{align}
The last inequality follows from the fact that $\int_{\R}v^p\,\dd y$ is a Lyapunov functional for equation \eqref{2.eq:v_p} \cite{MMS09}.
Since the relative R\'enyi entropy is dilation invariant, going back to the original variables \eqref{1.def:v} and \eqref{1.def:tau} we obtain
    \begin{align*}
        \renyi_p(u(t)|U_{p}(t)) \leq  -\frac{1}{2}\ln\left(1 - \left(1 - \exp({-2\renyi_p(u_0|B_p)})\right)\frac{\M_0}{\M(u(t))}\right)\,,\quad t > 0\,,
    \end{align*}
    where $U_{p}(t)$ is the rescaled Barenblatt profile having the same second moment as $u(t)$, i.e.
    \begin{equation}\label{3.def:Up}
        U_p(x,t) = \left(\frac{\M_t}{\M_0}\right)^{-1/2}\left(C_p - \frac{p-1}{p}\frac{\M_0}{\M_t}\frac{x^2}{2\sigma_t}\right)_+^{1/(p-1)}\,,
    \end{equation}
where $\M_t = \M(u(t))$ and $\sigma_t>0$ matches the second moment being equal to $\M_t$.
Finally, utilizing the Csiszar-Kullback inequality we conclude with stating the following theorem.
\begin{theorem}\label{3.tm}
    Let $1<p\leq 3/2$ and let $u$ be a smooth non-negative solution of the Cauchy problem \eqref{1.eq:lnp_sym} with initial data $u_0$ satisfiying
    \begin{equation*}
        \int_\R u_0\,\dd x = 1\,,\quad \int_\R  xu_0\,\dd x = 0\quad \text{and}\quad \int_\R x^2u_0\,\dd x = \M_0\,,
    \end{equation*}
    and $\renyi_p(u_0) < +\infty$. Then, there exists a constant $C>0$ independent of $u$, such that
    \begin{equation*}
        \|u(t) - U_p(t)\|_{L^1(\R)} \leq C\left(\renyi_p(u(t)|U_{p}(t))\right)^{1/2}\,\quad \text{for all }\ t>0\,,
    \end{equation*}
    where $U_p$ is defined in \eqref{3.def:Up}.
\end{theorem}

\begin{remark}
The obtained decay estimate of the relative Newmann-Ralston entropy in \eqref{3.ineq:NR_ren} is super-exponential and thus formally improves the exponential decay obtained for weak solutions in \cite{MMS09} (cf.~Figure \ref{fig:decay}). However, these improvements are only for early and intermediate times $\tau$, while for large times $\tau$, decay estimate \eqref{3.ineq:NR_ren} behaves exponentially. This can be seen from the right-hand side in \eqref{3.ineq:rendiss4}. Namely, for large times $\renyi_p(v(\tau)|B_p) = z$ will be small, and for small values $z$ we have $\exp(2z)-1\approx 2z$, which yields the exponential decay.
\end{remark}
\begin{remark}
    While it is possible to specify minimal regularity conditions under which the above formal calculations would be justified, the existence of corresponding solutions for $1\leq p < 3/2$ remains open in the literature. For clarity of exposition, we omit this issue here; it will be the focus of future work.
\end{remark}

\subsection{The thin-film equation} \label{sec:thin-film}

In this subsection we will prove that the dissipation of the relative R\'enyi entropy of order $3/2$ can be made rigorous along strong solutions of the rescaled thin-film equation. Prior to that we recall some basic analytical facts about the thin-film equation.

\subsubsection{Strong solutions: definition, basic properties and long-time behavior}\label{sec:CaTo}
We briefly review some basic properties of strong solutions to the thin-film equation \eqref{1.eq:tfe}, emphasizing the global existence \cite{BBD95, Ber96, BePu96} and the long-time behavior \cite{Ber96, CaTo02}. 

To start with, let us introduce the standard notation in the theory of thin-film equations. Let $T>0$ denote a given time horizon, $Q_T = \R\times(0,T)$ and $Q = \R\times(0,+\infty)$ are space-time cylinders, $P_T = Q_T\setminus \left(\{u=0\}\cup \{t=0\}\right)$ and $P = Q\setminus \left(\{u=0\}\cup \{t=0\}\right)$.
\begin{definition}
A nonnegative function $u(x,t)$ is called a \emph{weak solution} to the Cauchy problem \eqref{1.eq:tfe}-\eqref{1.eq:id} if $u$ satisfies
\begin{align*}
    u\in C^{1/2,1/8}(\overline{Q})\cap L^\infty(0,\infty;H^1(\R))\cap C^\infty(P)\,,\\
    u^{1/2}u_{xxx}\in L^2(P)
    \quad \text{and}\quad\int_Q u\psi_t + \int_P u u_{xxx}\psi_x = 0
\end{align*}
for all Lipschitz functions $\psi$ with compact support inside $Q$. Moreover, 
\begin{equation*}
    u(\cdot,0) = u_0\quad\text{and}\quad u_x(\cdot,t)\to u_{0x} \text{ strongly in }L^2(\R) \text{ as }t\downarrow 0\,.
\end{equation*}
\end{definition}
\noindent The regularity of weak solutions has been studied in \cite{BBD95, BePu96} and in particular they proved the existence of strong solutions.
\begin{definition}
A weak solution to Cauchy problem \eqref{1.eq:tfe}-\eqref{1.eq:id} which additionally satisfies 
\begin{equation*}
    u\in L^2(0,T;H^2(\R))
\end{equation*}
for any $T>0$ is called \emph{strong solution}.
\end{definition}
\noindent The existence of strong solutions to the Cauchy problem \eqref{1.eq:tfe}-\eqref{1.eq:id} with initial data of the finite second moment was proved in \cite{CaTo02}.  


Strong solutions satisfy additional regularity \cite{Ber96, CaTo02}: for any $T>0$ and $0 < s <1/2$, 
\begin{align*}
    u^{1-s/2} &\in L^2(0,T;H^2(\R))\,,\\
    (u^r)_x & \in L^4(Q_T) \quad\text{for any }1/2 - s/4 \leq r <1\,,
\end{align*}
and equation \eqref{1.eq:tfe} is satisfied in the following sense
\begin{equation*}
\int_Q u\psi_t - \int_Q uu_{xx}\psi_{xx} - \int_Q u^{1-r}\left(\frac{u^r}{r}\right)_xu_{xx}\psi_x = 0
\end{equation*}
for all $\psi\in C_c^\infty(Q)$. Furthermore, since $u(\cdot,t)\in C^1(\R)$ for a.e.~$t$, it follows that $u_x(\cdot,t)=0$ on the boundary of the support of $u$.
Strong solutions also preserve mass, i.e.
\begin{equation*}
    \int_\R u(t)\,\dd x = \int_\R u_0\,\dd x \quad \text{for all $t>0$}\,,
\end{equation*}
and dissipate the surface-tension energy 
$$\E(u) = \frac12\int_{\R}u_x^2\dd x\,$$ according to
\begin{align*}
    \E(u(T))+\int_{P_T} u u_{xxx}^2 \leq \E(u_0)\,,
\end{align*}
for any $T>0$. In particular, for strong solutions we have equality between the generalized Fisher information and the surface-tension energy, i.e.~for a.e.~$t$ it holds $\I_{3/2}(u(t)) = \E(u(t))$. 
Furthermore, strong solutions dissipate the $\alpha$-functionals, often called entropies, in the following way
\begin{align*}
    \frac{1}{\alpha(\alpha-1)}\frac{\dd}{\dd t}\int_{\R}u^\alpha\, \dd x = - \int_{\R\cap\{u>0\}}u^{\alpha-1}u_{xx}^2\,\dd x - \frac{(\alpha-1)(2-\alpha)}{3}\int_{\R\cap\{u>0\}}u^{\alpha-3}u_{x}^4\,\dd x\,,
\end{align*}
for all $1/2 < \alpha < 2$ and $\alpha\neq 1$. The case of $\alpha=1$ corresponds to the dissipation of the Boltzmann entropy (with the opposite sign)
\begin{equation*}
    \frac{\dd}{\dd t}\int_\R u \ln u\, \dd x = -\int_\R u_{xx}^2\, \dd x \leq 0\,.
\end{equation*}
Finally, the following integration by parts formula holds
\begin{equation}\label{3.eq:intbyparts}
    \int_{\R\cap\{u>0\}}u^{\alpha-1}u_{xx}u_x^2\, \dd x = \frac{1-\alpha}{3} \int_{\R\cap\{u>0\}}u^{\alpha-2}u_x^4\, \dd x\qquad \text{for a.e. }t>0\,.
\end{equation}

The long-term behavior of strong solutions to the Cauchy problem \eqref{1.eq:tfe}-\eqref{1.eq:id} has been first studied by Bernis in \cite{Ber96}. He showed that the surface-tension energy $\E(u(t))$ decays as $O(t^{-3/5})$ and the $L^\infty$-norm of solutions decays as $O(t^{-1/5})$ for $t\to \infty$. These decay rates turn out to be optimal, as they exactly match those of source-type solutions found by Smyth and Hill \cite{SmHi88}. 
The latter was improved in \cite[Theorem 5.1]{CaTo02} by showing that for a non-negative initial data $u_0\in H^1(\R)$ of vanishing first and finite second moment, there exists a constant $C>0$ such that
        \begin{equation*}
            \|u(x,t) - U(x,t)\|_{L^1(\R)} \leq C t^{-1/5}\,\quad\text{for all $t>1$}\,,
        \end{equation*}
where $U(x,t)$ is the unique strong self-similar solution to \eqref{1.eq:tfe} defined by \eqref{1.def:sssts}. The constant $C$ was explicitly calculated and in fact the result was given for the non-vanishing first moment.


\subsubsection{Justification of Theorem \ref{3.tm} for $p=3/2$}
The second moment $\M(u(t))$ is absolutely continuous function on $[0,+\infty)$ and its evolution along solutions to \eqref{1.eq:tfe} equals 
\begin{align*}
    \frac{\dd }{\dd t}\M(u(t)) &= 
     3\int_\R u_{x}^2\dd x= 6\, \E(u(t))\,.
\end{align*}
Rescaling a strong solution $u$ via the change of variables \eqref{1.def:v} and \eqref{1.def:tau} leads to the nonlocal thin-film equation
\begin{equation}\label{3.eq:v_tf}
        \pa_\tau v = -\frac{1}{\E(v)}\left(v^{3/2}\left(v^{1/2}-1\right)_{yy}\right)_{yy} + \frac{3}{\M_0}(y v)_y\,,
    \end{equation}
with strong solutions $v(y,\tau)$ of analogous regularity properties (cf.~\cite{CaTo02}).  
Let us justify the evolution of the relative R\'enyi entropy of order $3/2$,
\begin{align}\label{3.eq:rendiss_tf}
    \frac{\dd}{\dd \tau}\renyi_{3/2}(v(\tau)|B_{3/2}) &=  -\frac{3}{2\int_\R v^{3/2}\, \dd y} \frac{\K_{3/2}(v)}{\E(v)}  + \frac{3}{\M_0}\,,
\end{align}
where
\begin{equation*}
    \K_{3/2}(v) = 4\int_\R v^{3/2}\left(v^{1/2}\right)_{yy}^2\, \dd y\,.
\end{equation*}
First, we argue that
\begin{equation}\label{3.eq:KK}
    \K_{3/2}(v) = \int_{\R}v^{1/2}v_{yy}^2\, \dd y + \frac{1}{12}\int_{\R\cap\{v>0\}}v^{-3/2}v_y^4\, \dd y\,.
\end{equation}
Integrals on the right-hand side are finite for strong solutions and the equality follows by the integration by parts formula \eqref{3.eq:intbyparts} for $\alpha=1/2$. Second, the following integration by parts formula holds for strong solutions \cite{CaTo02},
\begin{align*}
    -\frac{3}{2} \int_\R y v^{1/2}v_y \,\dd y
    = \int_\R v^{3/2}\,\dd y\,,
\end{align*}
which justifies the calculation of the integral of the convective term.
Therefore, \eqref{3.eq:rendiss_tf} holds for strong solutions of equation \eqref{3.eq:v_tf}. Since Lemma \ref{2.lem:villp} is also valid for the strong solutions in the sense of \eqref{3.eq:KK}, we conclude that Theorem \eqref{3.tm} holds for strong solutions of the Cauchy problem \eqref{1.eq:tfe}-\eqref{1.eq:id}.


\subsection{The DLSS equation}\label{sec:dlss} For the DLSS equation, we also carry out formal calculations. It is known from the seminal work \cite{BLS94} that the global existence of smooth positive solutions is equivalent to the preservation of positivity. Under periodic boundary conditions, a sufficient condition ensuring such solutions—based on the smallness of the initial data—was established in \cite{CCT05}. For the Cauchy problem on the real line, however, this question remains open.

The rescaled DLSS equation reads
\begin{equation}\label{3.eq:dlss_res}
        \pa_\tau v = -\frac{1}{\I(v)}\left(v\left(\ln v\right)_{yy}\right)_{yy} + \frac{2}{\M_0}(y v)_y\,,
    \end{equation}
where $\I(v) = \I_1(v)$ denotes the Fisher information.
For $p=1$, the relative R\'enyi entropy becomes the relative Boltzmann entropy
\begin{equation*}
    \entropy(f|B_{1,f}) = \entropy(B_{1,f}) - \entropy(f) = \frac{1}{2}\ln(2\pi e\M_0) - \entropy(f)\,,
\end{equation*}
where 
\begin{equation*}
    \entropy(f) = -\int_{\R}f\ln f\, \dd x\,.
\end{equation*}
Then the evolution of the relative Boltzmann entropy along smooth, positive solutions to \eqref{3.eq:dlss_res} is given by
\begin{align}\label{3.eq:diss_relB}
    \frac{\dd}{\dd \tau}\entropy(v(\tau)|B_1) = -\frac{1}{\I(v)} \int_\R v\left(\ln v\right)_{yy}^2\, \dd y + \frac{2}{\M_0}\,.
\end{align}
Using the inequality \eqref{2.ineq:KIp} (Lemma \ref{2.lem:villp}) and the fact that
\begin{equation*}
    \frac{2}{\M_0} = 4 \I(B_1)
\end{equation*}
we estimate \eqref{3.eq:diss_relB} as
\begin{align}\nonumber
    \frac{\dd}{\dd \tau}\entropy(v(\tau)|B_1) &\leq -4\I(v) + 4\I(B_1) = -\frac{2}{\M_0}\left(\frac{\I(v)}{\I(B_1)} - 1\right).\label{3.ineq:rendiss3}
\end{align}
Using the entropy power inequality \eqref{2.ineq:epi_p} (Lemma \ref{2.lem:epip}) and identity \eqref{2.eq:expreny} it follows 
\begin{equation*}
    \frac{\I(v)}{\I(B_1)}\geq \exp(2\entropy(v|B_1))\,.
\end{equation*} 
Therefore,
\begin{align*}
    \frac{\dd}{\dd \tau}\entropy(v(\tau)|B_1) &\leq -\frac{2}{\M_0}\left(\exp(2\entropy(v|B_1)) - 1\right)\,.
\end{align*}
As before, this differential inequality can be solved explicitly by separation of variables and one obtains
\begin{equation*}
    \entropy(v(\tau)|B_1) \leq -\frac{1}{2}\ln\left(1 - \left(1 - \exp({-2\entropy(v_0|B_1)})\right)\exp\left(-\frac{4}{\M_0}\tau\right)\right)\,,\quad \tau > 0\,.
\end{equation*}
Going back to the original variables \eqref{1.def:v} and \eqref{1.def:tau} we obtain
    \begin{align*}
        \entropy(u(t)|U_{1}(t)) \leq  -\frac{1}{2}\ln\left(1 - \left(1 - \exp({-2\entropy(u_0|B_p)})\right)\frac{\M_0}{\M(u(t))}\right)\,,\quad t > 0\,,
    \end{align*}
    where $U_{1}(t)$ is the rescaled Gaussian profile having the same second moment as $u(t)$, i.e.
    \begin{equation*}
        U_1(x,t) = \frac{1}{\sqrt{2\pi\M_t}}\exp\left(-\frac{x^2}{2\M_t}\right)\,.
    \end{equation*}
Finally, the Csisz\'ar-Kullback inequality asserts the following theorem.
\begin{theorem}\label{3.tm_dlss}
   Let $u$ be a smooth positive solution of the Cauchy problem for the DLSS equation \eqref{1.eq:dlss_log} with initial data $u_0$ satisfiying
    \begin{equation*}
        \int_\R u_0\,\dd x = 1\,,\quad \int_\R  xu_0\,\dd x = 0\quad \text{and}\quad \int_\R x^2u_0\,\dd x = \M_0\,,
    \end{equation*}
    and $\entropy(u_0) < +\infty$. Then, there exists a constant $C>0$ independent of $u$, such that
    \begin{equation*}
        \|u(t) - U_1(t)\|_{L^1(\R)} \leq C\left(\entropy(u(t)|U_{1}(t))\right)^{1/2}\,\quad \text{for all }\ t>0\,.
    \end{equation*}
\end{theorem}


\end{document}